\documentclass[a4paper,12pt]{amsart}
\usepackage{amssymb,amscd,amsmath,mathrsfs}

\usepackage{mathptmx}
\usepackage{dsfont}

\usepackage{latexsym}
\usepackage{tensor}
\usepackage{epsfig}
\usepackage{graphicx}
\usepackage[all]{xy}
\usepackage{graphics}
\usepackage{color}
\usepackage{tikz}

\theoremstyle{plain}
\newtheorem{theorem}{Theorem}[section]
\newtheorem{proposition}[theorem]{Proposition}
\newtheorem{corollary}[theorem]{Corollary}
\newtheorem{lemma}[theorem]{Lemma}

\theoremstyle{definition}

\newtheorem{remark}[theorem]{Remark}

\numberwithin{equation}{section}

\newcommand{\norm}[1]{\left\Vert{#1}\right\Vert}
\newcommand{\abs}[1]{\left\vert{#1}\right\vert}

\newcommand{\Natural}{\mathds{N}}
\newcommand{\Integer}{\mathds{Z}}

\title{Homfly polynomials for periodic knots via state model}

\author{Joonoh Kim}
\address{Department of Mathematics, Pusan National University, Busan, Korea}

\author{Kyoung-Tark Kim}
\address{Sogang research team for discrete and geometric structures, Sogang University, Mapo-gu, Seoul 04107, Korea}

\begin{document}

\maketitle

\begin{abstract}
We give criteria for oriented links to be periodic of prime order using the quantum $\mathrm{SL}(N)$-invariant.
The criteria are based upon an observation on the linking number between a periodic knot and its axis of the rotation.
\end{abstract}

\section{Introduction}\label{sec:intro}
\noindent A link $L$ in $S^{3}$ is called \textit{$p$-periodic} ($p \in \Integer_{\geq 2}$) provided that there is a homeomorphism $g: S^{3} \rightarrow S^{3}$ of order $p
$ whose fixed point set $\gamma$ is $S^1$ on $S^3$ with $\gamma \cap L = \emptyset$.
(See \cite[Section 1]{Mu}.)

The \textit{homfly polynomial} $P_{L}(a,z) \in \Integer [a^{\pm 1},z^{\pm 1}]$ of an oriented link $L$ is known to be defined uniquely by the following recursive relation:
\begin{enumerate}
\item $P_{T_1}(a, z) = 1$;
\item $a^{-1}P_{L_{+}}(a, z) -a P_{L_{-}}(a,z)=zP_{L_{0}}(a,z)$,
\end{enumerate}
where $T_1$ is the trivial knot, and $L_{+}$, $L_{-}$, and $L_{0}$ are obtained from $L$ which are identical with $L$ except one given crossing as depicted in Figure~\ref{fig:pmz} (see \cite[Theorem 8.2.6, p. 105]{Kaw}).
The \textit{Jones polynomial} $V_{L}(t)$ of an oriented link $L$ can be obtained from $P_{L}(a,z)$ by substituting $a=t$ and $z = \sqrt{t}-\frac{1}{\sqrt{t}}$ (see \cite[Proposition 8.2.8, p. 106]{Kaw}).

The \textit{Kauffman polynomial} $F_L(a,z)\in \Integer[a^{\pm 1}, z^{\pm 1}]$ of an oriented link $L$ is defined by
\[ F_L(a,z) = a^{-w(D)} \Lambda_D(a,z), \]
where $D$ is a diagram of $L$, $w(D)$ is the writhe of $D$, and the polynomial $\Lambda_D(a,z)$ is a regular isotopy invariant of an unoriented link obtained by:
\begin{enumerate}
\item $\Lambda_{\tikz[scale=1,x=1pt,y=1pt] \draw (0,0) circle (3);}(a,z) = 1$;
\item $\Lambda_{\tikz[scale=1,x=1pt,y=1pt] \draw (0,8)--(8,0) (0,0)--(3,3) (5,5)--(8,8);} (a,z) + \Lambda_{\tikz[scale=1,x=1pt,y=1pt] \draw (0,0)--(8,8) (0,8)--(3,5) (5,3)--(8,0);} (a,z) = z (\Lambda_{\tikz[scale=1,x=1pt,y=1pt] \draw[rounded corners=1pt] (0,8)--(3,5)--(5,5)--(8,8) (0,0)--(3,3)--(5,3)--(8,0);}(a,z) + \Lambda_{\tikz[scale=1,x=1pt,y=1pt] \draw[rounded corners=1pt] (0,0)--(3,3)--(3,5)--(0,8) (8,0)--(5,3)--(5,5)--(8,8);}(a,z))$;
\item $\Lambda_{\tikz[scale=1,x= 1pt,y=1pt] \draw[rounded corners=2pt] (0,8)--(7,4)--(4,1)--(1,4)--(3,4) (6,7)--(8,8);}(a,z) = a \Lambda_{\tikz[scale=1,x= 1pt,y=1pt] \draw[rounded corners=2pt] (0,8)--(3,5)--(5,5)--(8,8);}(a,z)$;
\item $\Lambda_{\tikz[scale=1,x=1pt,y=1pt] \draw[rounded corners=2pt] (8,8)--(1,4)--(4,1)--(7,4)--(5,4) (2,7)--(0,8);}(a,z) = a^{-1} \Lambda_{\tikz[scale=1,x= 1pt,y=1pt] \draw[rounded corners=2pt] (0,8)--(3,5)--(5,5)--(8,8);}(a,z)$.
\end{enumerate}

The following are several known results for periodic knots or links.
They can be used as criteria for non-periodicity of links.

\begin{theorem}[Przytycki {\cite[Theorem 1.2, Theorem 1.4]{Pr3}}]\label{thm:Przytycki}
Let $p$ be a prime number and $L$ an oriented $p$-periodic link.
Then
\begin{enumerate}
\item $P_{L}(a,z)\equiv P_{L}(a^{-1},z) \mbox{ mod } (p, z^{p})$;
\item $F_{L}(a,z)\equiv F_{L}(a^{-1},z) \mbox{ mod } (p, z^{p})$.
\end{enumerate}
\end{theorem}

\begin{theorem}[Traczyk {\cite{Tr}}, {\cite[Exercise 10.1.6, p. 125]{Kaw}}]\label{thm:Traczyk}
Let $L$ be an oriented $p$-periodic link.
Then $V_{L}(t)\equiv V_{L}(t^{-1}) \mbox{ mod } (p,~ t^{p}-1)$.
\end{theorem}

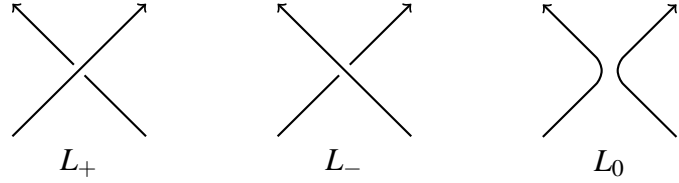
\begin{figure}
\begin{tabular}{ccccc}
\begin{tikzpicture}[scale=1,x=1pt,y=1pt]
\draw [->, thick] (0,0) -- (50,50);
\draw [thick] (50,0) -- (27,23);
\draw [->, thick] (23,27) -- (0,50);
\node at (25,-10) {$L_{+}$};
\end{tikzpicture}
&
\hspace{1cm}
&
\begin{tikzpicture}[scale=1,x=1pt,y=1pt]
\draw [->, thick] (50,0) -- (0,50);
\draw [thick] (0,0) -- (23,23);
\draw [->, thick] (27,27) -- (50,50);
\node at (25,-10) {$L_{-}$};
\end{tikzpicture}
&
\hspace{1cm}
&
\begin{tikzpicture}[scale=1,x=1pt,y=1pt]
\draw [->, thick, rounded corners=3pt] (0,0) -- (22,22) -- (22,28) -- (0,50);
\draw [->, thick, rounded corners=3pt] (50,0) -- (28,22) -- (28,28) -- (50,50);
\node at (25,-10) {$L_{0}$};
\end{tikzpicture}
\end{tabular}
\caption{Skein triple}
\label{fig:pmz}
\end{figure}

\begin{theorem}[Murasugi {\cite[Section 1]{Mu}}]\label{thm:Murasugi}
If $K$ is an oriented $p^r$-periodic knot ($p$ is a prime and $r\in \Natural$), then the Alexander polynomial $\Delta_K(t)$ of $K$ satisfies
\[ \Delta_{K}(t) \equiv f (t)^{p^{r}} (1 + t + t^{2}  + \cdots + t^{\lambda-1})^{p^{r}-1} \quad \mbox{ mod } p  \]
for some polynomial $f(t)$ and $\lambda \in \Integer_{>0}$ with $\mathrm{gcd}(\lambda, p)=1 $.
\end{theorem}

In the following theorem, Traczyk used the degree-zero term $P_0(a)$ in $z$ of $P_{K}(a,z) = \sum_{i \geq 0} P_{2i}(a) z^{2i}$ for periodic knots.

\begin{theorem}[Traczyk {\cite[Theorem 1.1]{Tr1}}, {\cite[Theorem 1 (a)]{Pr1}}]\label{thm:Traczyk}
If $K$ is a $p$-periodic knot ($p$ an odd prime) and the linking number of the periodic version of $K$ with the axis of the rotation is equal to $\lambda$, then in the polynomial $P_0(a) =\sum_{i \in \Integer} c_{2i}a^{2i}$ we have $c_{2i} \equiv c_{2i+2} \:\mathrm{mod}\: p$ for all $i \in \Integer$, except possibly when $2i + 1 \equiv \pm \lambda \mbox{ mod } p$.
\end{theorem}

\begin{remark}
The paper \cite{Tr1} used a different skein relation in \cite{LiMi}, namely,
\[ l P_{L_{+}}(l, m) + l^{-1} P_{L_{-}}(l,m) + mP_{L_{0}}(l,m) = 0. \]
Theorem~\ref{thm:Traczyk} is a modified version \cite{Pr1} in conformity with our definition of $P_L(a,z)$ in variables $a$ and $z$.
\end{remark}

The \textit{quantum $\mathrm{SL}(N)$-invariant $\mathcal{P}^{(N)}_L(q)$} of an oriented link $L$ (see \cite[Eq. (3), p. 5]{GS} and \cite[Remark in p. 171]{Ka}) is defined by
\begin{equation}\label{eq:defofqi}
\mathcal{P}^{(N)}_L(q) := \frac{q^{N}-q^{-N}}{q-q^{-1}} P_{L}(q^{-N}, q-q^{-1}) \in \Integer[q^{\pm 1}].
\end{equation}
In this paper we give criteria for periodic links using the quantum $\mathrm{SL}(N)$-invariant $\mathcal{P}^{(N)}_{L}(q)$ of an oriented link $L$.

\begin{figure}
\begin{tabular}{ccc}
\begin{tikzpicture}[scale=1,x=1pt,y=1pt]
\draw [->, thick] (0,0) -- (50,50);
\draw [thick] (50,0) -- (27,23);
\draw [->, thick] (23,27) -- (0,50);
\node at (25,38) {$+1$};
\end{tikzpicture}
&
\hspace{1.5cm}
&
\begin{tikzpicture}[scale=1,x=1pt,y=1pt]
\draw [->, thick] (50,0) -- (0,50);
\draw [thick] (0,0) -- (23,23);
\draw [->, thick] (27,27) -- (50,50);
\node at (25,38) {$-1$};
\end{tikzpicture}

\end{tabular}
\caption{Crossing signs}
\label{fig:cs}
\end{figure}
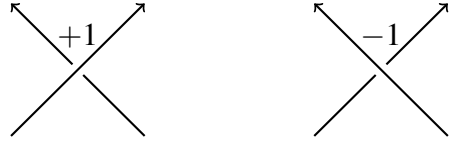

\section{Quantum $\mathrm{SL}(N)$-invariant and its quiver state model}\label{sec:suq}
\noindent Throughout the paper, the letters $L$ and $D$ always stand for a link and its diagram, respectively.

We recall a \textit{Yang--Baxter state model} of $\mathcal{P}^{(N)}_L(q)$ for an oriented link $L$.
(See e.g. \cite[Remark in p. 171]{Ka} for detailed theory.)
For $N \in \Integer_{\geq 2}$, we set
\[ \mathcal{I}_{N} := \{-N+1, -N+3, \ldots, N-3, N-1\}. \quad (\abs{\mathcal{I}_N} = N) \]
Then we obtain the following description of $\mathcal{P}^{(N)}_L(q)$:
\[ \mathcal{P}^{(N)}_L(q) = q^{-w(D)\cdot N}\langle D \rangle. \]
Here $w(D)$ is the \textit{writhe} of $D$ ($=$the sum of \textit{crossing signs} in $D$ as shown in Figure~\ref{fig:cs}) and $\langle D \rangle$ is defined by
\[ \langle D \rangle := \sum_{\sigma}\langle D | \sigma\rangle q^{\|\sigma\|}, \quad (\sigma \mbox{ runs over all states on } D) \]
where a \textit{state} $\sigma$, the product $\langle D|\sigma\rangle$ of \textit{vertex weights}, and the \textit{state norm} $\|\sigma\|$ will be explained now.

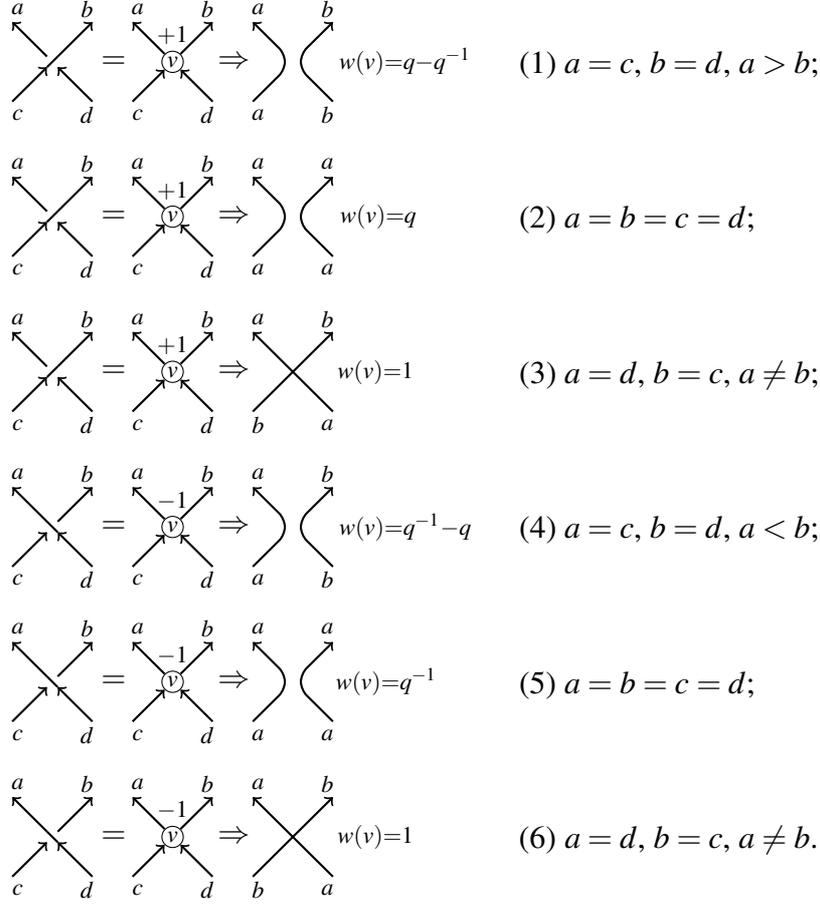
\begin{figure}
\begin{tabular}{ll}

\begin{tikzpicture}[scale=1,x=1pt,y=1pt]
\draw [->, thick] (13,17) -- (0,30);
\draw [->, thick] (17,17) -- (30,30);
\draw [->, thick] (0,0) -- (13,13);
\draw [->, thick] (30,0) -- (17,13);
\draw [-, thick] (13,13) -- (17,17);

\node at (2,35) {$\scriptstyle a$};
\node at (28,35) {$\scriptstyle b$};
\node at (2,-5) {$\scriptstyle c$};
\node at (28,-5) {$\scriptstyle d$};

\node at (38,15) {$=$};

\begin{scope}[xshift=45pt]
\draw (15,15) circle (4);
\draw [->, thick] (0,0) -- (12,12);
\draw [->, thick] (18,18) -- (30,30);
\draw [->, thick] (30,0) -- (18,12);
\draw [->, thick] (12,18) -- (0,30);
\node at (15,15) {$\scriptstyle v$};
\node at (15,25) {$\scriptstyle +1$};

\node at (2,35) {$\scriptstyle a$};
\node at (28,35) {$\scriptstyle b$};
\node at (2,-5) {$\scriptstyle c$};
\node at (28,-5) {$\scriptstyle d$};

\node at (38,15) {$\Rightarrow$};
\end{scope}

\begin{scope}[xshift=90pt]
\draw [->, thick, rounded corners=3pt] (0,0) -- (12,12) -- (12,18) -- (0,30);
\draw [->, thick, rounded corners=3pt] (30,0) -- (18,12) -- (18,18) -- (30,30);

\node at (2,35) {$\scriptstyle a$};
\node at (28,35) {$\scriptstyle b$};
\node at (2,-5) {$\scriptstyle a$};
\node at (28,-5) {$\scriptstyle b$};

\node at (57,15) {$\scriptstyle w(v) = q - q^{-1}$};
\end{scope}
\end{tikzpicture}

&

\begin{tikzpicture}[scale=1,x=1pt,y=1pt]
\node at (0,0) {};
\node at (0,22) {(1) $a=c$, $b=d$, $a>b$;};
\end{tikzpicture}

\\

\begin{tikzpicture}[scale=1,x=1pt,y=1pt]
\draw [->, thick] (13,17) -- (0,30);
\draw [->, thick] (17,17) -- (30,30);
\draw [->, thick] (0,0) -- (13,13);
\draw [->, thick] (30,0) -- (17,13);
\draw [-, thick] (13,13) -- (17,17);

\node at (2,35) {$\scriptstyle a$};
\node at (28,35) {$\scriptstyle b$};
\node at (2,-5) {$\scriptstyle c$};
\node at (28,-5) {$\scriptstyle d$};

\node at (38,15) {$=$};

\begin{scope}[xshift=45pt]
\draw (15,15) circle (4);
\draw [->, thick] (0,0) -- (12,12);
\draw [->, thick] (18,18) -- (30,30);
\draw [->, thick] (30,0) -- (18,12);
\draw [->, thick] (12,18) -- (0,30);
\node at (15,15) {$\scriptstyle v$};
\node at (15,25) {$\scriptstyle +1$};

\node at (2,35) {$\scriptstyle a$};
\node at (28,35) {$\scriptstyle b$};
\node at (2,-5) {$\scriptstyle c$};
\node at (28,-5) {$\scriptstyle d$};

\node at (38,15) {$\Rightarrow$};
\end{scope}

\begin{scope}[xshift=90pt]
\draw [->, thick, rounded corners=3pt] (0,0) -- (12,12) -- (12,18) -- (0,30);
\draw [->, thick, rounded corners=3pt] (30,0) -- (18,12) -- (18,18) -- (30,30);

\node at (2,35) {$\scriptstyle a$};
\node at (28,35) {$\scriptstyle a$};
\node at (2,-5) {$\scriptstyle a$};
\node at (28,-5) {$\scriptstyle a$};

\node at (47,15) {$\scriptstyle w(v) = q$};
\end{scope}
\end{tikzpicture}

&

\begin{tikzpicture}[scale=1,x=1pt,y=1pt]
\node at (0,0) {};
\node at (0,22) {(2) $a=b=c=d$;};
\end{tikzpicture}

\\

\begin{tikzpicture}[scale=1,x=1pt,y=1pt]
\draw [->, thick] (13,17) -- (0,30);
\draw [->, thick] (17,17) -- (30,30);
\draw [->, thick] (0,0) -- (13,13);
\draw [->, thick] (30,0) -- (17,13);
\draw [-, thick] (13,13) -- (17,17);

\node at (2,35) {$\scriptstyle a$};
\node at (28,35) {$\scriptstyle b$};
\node at (2,-5) {$\scriptstyle c$};
\node at (28,-5) {$\scriptstyle d$};

\node at (38,15) {$=$};

\begin{scope}[xshift=45pt]
\draw (15,15) circle (4);
\draw [->, thick] (0,0) -- (12,12);
\draw [->, thick] (18,18) -- (30,30);
\draw [->, thick] (30,0) -- (18,12);
\draw [->, thick] (12,18) -- (0,30);
\node at (15,15) {$\scriptstyle v$};
\node at (15,25) {$\scriptstyle +1$};

\node at (2,35) {$\scriptstyle a$};
\node at (28,35) {$\scriptstyle b$};
\node at (2,-5) {$\scriptstyle c$};
\node at (28,-5) {$\scriptstyle d$};

\node at (38,15) {$\Rightarrow$};
\end{scope}

\begin{scope}[xshift=90pt]
\draw [->, thick] (0,0) -- (30,30);
\draw [->, thick] (30,0) -- (0,30);

\node at (2,35) {$\scriptstyle a$};
\node at (28,35) {$\scriptstyle b$};
\node at (2,-5) {$\scriptstyle b$};
\node at (28,-5) {$\scriptstyle a$};

\node at (46,15) {$\scriptstyle w(v) = 1$};
\end{scope}
\end{tikzpicture}

&

\begin{tikzpicture}[scale=1,x=1pt,y=1pt]
\node at (0,0) {};
\node at (0,22) {(3) $a=d$, $b=c$, $a\neq b$;};
\end{tikzpicture}

\\

\begin{tikzpicture}[scale=1,x=1pt,y=1pt]
\draw [->, thick] (13,17) -- (0,30);
\draw [->, thick] (17,17) -- (30,30);
\draw [->, thick] (0,0) -- (13,13);
\draw [->, thick] (30,0) -- (17,13);
\draw [-, thick] (13,17) -- (17,13);

\node at (2,35) {$\scriptstyle a$};
\node at (28,35) {$\scriptstyle b$};
\node at (2,-5) {$\scriptstyle c$};
\node at (28,-5) {$\scriptstyle d$};

\node at (38,15) {$=$};

\begin{scope}[xshift=45pt]
\draw (15,15) circle (4);
\draw [->, thick] (0,0) -- (12,12);
\draw [->, thick] (18,18) -- (30,30);
\draw [->, thick] (30,0) -- (18,12);
\draw [->, thick] (12,18) -- (0,30);

\node at (15,15) {$\scriptstyle v$};
\node at (15,25) {$\scriptstyle -1$};

\node at (2,35) {$\scriptstyle a$};
\node at (28,35) {$\scriptstyle b$};
\node at (2,-5) {$\scriptstyle c$};
\node at (28,-5) {$\scriptstyle d$};

\node at (38,15) {$\Rightarrow$};
\end{scope}

\begin{scope}[xshift=90pt]
\draw [->, thick, rounded corners=3pt] (0,0) -- (12,12) -- (12,18) -- (0,30);
\draw [->, thick, rounded corners=3pt] (30,0) -- (18,12) -- (18,18) -- (30,30);

\node at (2,35) {$\scriptstyle a$};
\node at (28,35) {$\scriptstyle b$};
\node at (2,-5) {$\scriptstyle a$};
\node at (28,-5) {$\scriptstyle b$};

\node at (57,15) {$\scriptstyle w(v) = q^{-1} - q$};
\end{scope}
\end{tikzpicture}

&

\begin{tikzpicture}[scale=1,x=1pt,y=1pt]
\node at (0,0) {};
\node at (0,22) {(4) $a=c$, $b=d$, $a<b$;};
\end{tikzpicture}

\\

\begin{tikzpicture}[scale=1,x=1pt,y=1pt]
\draw [->, thick] (13,17) -- (0,30);
\draw [->, thick] (17,17) -- (30,30);
\draw [->, thick] (0,0) -- (13,13);
\draw [->, thick] (30,0) -- (17,13);
\draw [-, thick] (13,17) -- (17,13);

\node at (2,35) {$\scriptstyle a$};
\node at (28,35) {$\scriptstyle b$};
\node at (2,-5) {$\scriptstyle c$};
\node at (28,-5) {$\scriptstyle d$};

\node at (38,15) {$=$};

\begin{scope}[xshift=45pt]
\draw (15,15) circle (4);
\draw [->, thick] (0,0) -- (12,12);
\draw [->, thick] (18,18) -- (30,30);
\draw [->, thick] (30,0) -- (18,12);
\draw [->, thick] (12,18) -- (0,30);
\node at (15,15) {$\scriptstyle v$};
\node at (15,25) {$\scriptstyle -1$};

\node at (2,35) {$\scriptstyle a$};
\node at (28,35) {$\scriptstyle b$};
\node at (2,-5) {$\scriptstyle c$};
\node at (28,-5) {$\scriptstyle d$};

\node at (38,15) {$\Rightarrow$};
\end{scope}

\begin{scope}[xshift=90pt]
\draw [->, thick, rounded corners=3pt] (0,0) -- (12,12) -- (12,18) -- (0,30);
\draw [->, thick, rounded corners=3pt] (30,0) -- (18,12) -- (18,18) -- (30,30);

\node at (2,35) {$\scriptstyle a$};
\node at (28,35) {$\scriptstyle a$};
\node at (2,-5) {$\scriptstyle a$};
\node at (28,-5) {$\scriptstyle a$};

\node at (50,15) {$\scriptstyle w(v) = q^{-1}$};
\end{scope}
\end{tikzpicture}

&

\begin{tikzpicture}[scale=1,x=1pt,y=1pt]
\node at (0,0) {};
\node at (0,22) {(5) $a=b=c=d$;};
\end{tikzpicture}

\\

\begin{tikzpicture}[scale=1,x=1pt,y=1pt]
\draw [->, thick] (13,17) -- (0,30);
\draw [->, thick] (17,17) -- (30,30);
\draw [->, thick] (0,0) -- (13,13);
\draw [->, thick] (30,0) -- (17,13);
\draw [-, thick] (13,17) -- (17,13);

\node at (2,35) {$\scriptstyle a$};
\node at (28,35) {$\scriptstyle b$};
\node at (2,-5) {$\scriptstyle c$};
\node at (28,-5) {$\scriptstyle d$};

\node at (38,15) {$=$};

\begin{scope}[xshift=45pt]
\draw (15,15) circle (4);
\draw [->, thick] (0,0) -- (12,12);
\draw [->, thick] (18,18) -- (30,30);
\draw [->, thick] (30,0) -- (18,12);
\draw [->, thick] (12,18) -- (0,30);
\node at (15,15) {$\scriptstyle v$};
\node at (15,25) {$\scriptstyle -1$};

\node at (2,35) {$\scriptstyle a$};
\node at (28,35) {$\scriptstyle b$};
\node at (2,-5) {$\scriptstyle c$};
\node at (28,-5) {$\scriptstyle d$};

\node at (38,15) {$\Rightarrow$};
\end{scope}

\begin{scope}[xshift=90pt]
\draw [->, thick] (0,0) -- (30,30);
\draw [->, thick] (30,0) -- (0,30);

\node at (2,35) {$\scriptstyle a$};
\node at (28,35) {$\scriptstyle b$};
\node at (2,-5) {$\scriptstyle b$};
\node at (28,-5) {$\scriptstyle a$};

\node at (46,15) {$\scriptstyle w(v) = 1$};
\end{scope}
\end{tikzpicture}

&

\begin{tikzpicture}[scale=1,x=1pt,y=1pt]
\node at (0,0) {};
\node at (0,22) {(6) $a=d$, $b=c$, $a\neq b$.};
\end{tikzpicture}

\end{tabular}
\caption{Rules for splice and projection with vertex weights}
\label{fig:splcprojhomfly}
\end{figure}

Fix a diagram $D$ of $L$.
We first see that $D$ corresponds to a planar quiver ($=$directed multigraph) together with crossing data, say
\[ G(D) = \left( V(D), A(D), C(D) \right), \]
where the vertex set $V(D)$ is the set of crossings in $D$, the arrow set $A(D)$ is the set of oriented arcs in $D$, and $C(D)$ is the function $V(D) \rightarrow \{ \pm 1 \}$ for crossing signs given as in Figure~\ref{fig:cs}.
Note that every vertex in $V(D)$ has indegree $2$ and outdegree $2$.
Now fix $N \in \Integer_{\geq 2}$.
A \textit{state} (or an \textit{$N$-state}) on $D$ is a pair $(G(D), \eta)$ such that $\eta : A(D) \rightarrow \mathcal{I}_N$ is an arrow labeling (or coloring) function satisfying the following \emph{local condition}:
\begin{quotation}
For each $v \in V(D)$, its adjacent arrow-labels $a,b,c,d \in \mathcal{I}_N$ fulfill one of the six conditions in Figure~\ref{fig:splcprojhomfly}.
\end{quotation}
If $\sigma = (G(D), \eta)$ is a state on $D$, then we remove each vertex $v\in V(D)$ by the rules in Figure~\ref{fig:splcprojhomfly} so as to obtain the \textit{vertex weight} $w(v)$ of $v$.
The quantity $\langle D | \sigma \rangle$ is defined by
\[ \langle D | \sigma \rangle := \prod_{v \in V(D)} w(v) \in \Integer[q^{\pm 1}]. \]
If every vertices are removed according to the rules in Figure~\ref{fig:splcprojhomfly}, then we get a diagram of planar loops with some flat crossings (graphical crossings).
It turns out (easy to verify) that the resulting diagram consists entirely of \emph{simple closed} oriented loops, called \textit{component loops} of $\sigma$.
The \textit{norm} $\norm{\sigma}$ of $\sigma$ is now defined by
\[ \norm{\sigma} := \sum_{\ell} \mathrm{label}(\ell) \cdot \mathrm{rot}(\ell), \]
where the summation is over all component loops $\ell$ of $\sigma$, $\mathrm{label}(\ell)$ is the assigned (coherent) arrow-label for $\ell$, and $\mathrm{rot}(D)$ is defined by
\[ \mathrm{rot}(\ell) := \left\{
\begin{array}{ll}
  +1 & \mbox{ if the orientation of $\ell$ is counterclockwise;} \\
  -1 & \mbox{ if the orientation of $\ell$ is clockwise.} \\
\end{array}\right. \]
The definition of $\mathrm{rot}(\ell)$ makes sense because $\ell$ is simply closed.

\section{A congruence of $\mathcal{P}^{(N)}_L(q)$ for periodic links}\label{sec:congqpm1}
\noindent In what follows we occasionally use the next easy lemma.

\begin{lemma}\label{lem:exp}
Let $i,j \in \Integer$.
If $i \equiv j \mbox{ mod } p$, then $q^{i} \equiv q^{j} \mbox{ mod } (q^{p}-1)$.
\end{lemma}

\begin{proof}
Assume that $i>j$ with $i=j+kp$ for some positive integer $k$.
Then $q^{i}-q^{j} = ((q^{p})^{k}-1) q^{j} = (q^{p}-1)(q^{p(k-1)} + \cdots + q^{p}+ 1) q^{j}$.
\end{proof}

The following is our main congruence of $\mathcal{P}^{(N)}_L(q)$ for a periodic link.

\begin{theorem}\label{thm:maincongm}
Let $p$ be a prime.
Suppose that $L = \bigcup_{j=1}^{m} L_j $ is a $p$-periodic oriented link of $m$ components with diagram $D = \bigcup_{j=1}^{m} D_j$ and the axis $\gamma$ of the rotation.
Then we have
\[ \mathcal{P}^{(N)}_L(q) \equiv \sum_{\varphi \in \mathcal{I}_N^{m}} \prod_{j=1}^{m} q^{\lambda_j \cdot \varphi(j)} \quad \mbox{ mod } (p, q^p -1), \]
where $\mathcal{I}_N^{m}$ is the set of all functions from $\{ 1, \ldots, m \}$ into $\mathcal{I}_N$, and $\lambda_j$ is the linking number between $L_j$ and $\gamma$.
\end{theorem}

A criterion for non-periodicity is an immediate corollary:

\begin{corollary}\label{cor:crilinkm}
Suppose $p$ is a prime and $L = \bigcup_{j=1}^{m} L_j$ is an oriented link.
If
\[ \mathcal{P}^{(N)}_L(q) \not\equiv \sum_{\varphi \in \mathcal{I}_N^m} \prod_{j=1}^m q^{\psi(j) \cdot \varphi(j)} \quad \mbox{ mod } (p,q^p -1) \]
for any function $\psi : \{ 1, \ldots, m \} \rightarrow \{ 0, \ldots, p-1 \}$, then $L$ is not $p$-periodic.
\end{corollary}

\begin{proof}[Proof of Theorem~\ref{thm:maincongm}]
We denote by $\zeta$ the rotation around $\gamma$ through the angle $2\pi / p$, and assume without loss of generality that $D$ and $G(D)$ are $p$-periodic configurations, i.e., symmetric under $\zeta$.

Let $\sigma$ be an $N$-state on $D$.
Then either $\zeta(\sigma) \neq \sigma$ or $\zeta(\sigma) = \sigma$.
If $\sigma$ is not symmetric under $\zeta$ (i.e., $\zeta(\sigma) \neq \sigma$), then the orbit of $\sigma$ under the action of $\zeta$ consists of $p$ distinct but congruent states $\sigma, \zeta(\sigma), \ldots, \zeta^{p-1}(\sigma)$.
Since all these congruent states contribute the same value toward $\langle D \rangle$, we have $\sum_{i=0}^{p-1} \langle D | \zeta^i(\sigma) \rangle q^{\norm{\sigma}} \equiv 0 \mbox{ mod } p$.
If $\sigma$ is symmetric under $\zeta$ (i.e., $\zeta(\sigma) = \sigma$) and $\sigma$ has a vertex of weight $\pm(q-q^{-1})$, then $\langle D | \sigma \rangle$ has a factor $\pm(q-q^{-1})^{p}$ whence zero modulo $(p, q^{p}-1)$.
On the other hand, if $\sigma$ is symmetric under $\zeta$ and $\sigma$ has no vertex of weight $\pm(q-q^{-1})$ (i.e., each vertex of $\sigma$ has weight $q^{\pm 1}$ or $1$), then $\langle D | \sigma \rangle = q^{pk}$ for some $k \in \Integer$, and hence $\langle D | \sigma \rangle \equiv 1 \mbox{ mod } (q^{p}-1)$.

When $\sigma$ has no vertex of weight $\pm(q-q^{-1})$, we say that $\sigma$ is ``\textit{proper}''. 
From the previous argument we see that
\begin{equation}\label{eq:symprop}
\langle D \rangle = \sum_{\sigma \, : \, \mathrm{all}} \langle D | \sigma \rangle q^{\norm{\sigma}} \equiv \sum_{\sigma \, : \, \mathrm{proper}\:\: \& \atop \quad\mathrm{symmetric}} q^{\norm{\sigma}} \quad\mbox{ mod } (p, q^p -1).
\tag{$\dagger$}
\end{equation}

Next we calculate the norm of a proper symmetric state.
Let $C_\sigma$ be the set of component loops of a state $\sigma$.
We consider the subset $C_\sigma^{\gamma} := \{ \ell \in C_\sigma \, : \, \mbox{ there is $\gamma$ inside $\ell$ }\}$.
If $\sigma$ is symmetric under $\zeta$, then
\[ \norm{\sigma} = \sum_{\ell \in C_\sigma} \mathrm{label}(\ell)\cdot\mathrm{rot}(\ell) \equiv \sum_{\ell \in C_\sigma^\gamma} \mathrm{label}(\ell)\cdot\mathrm{rot}(\ell) \quad \mbox{ mod } p. \]
Here we note from Lemma~\ref{lem:exp} that the quantity $q^{\norm{\sigma}}$ is determined up to modulo $(q^p -1)$ whenever the exponent $\norm{\sigma}$ is determined up to modulo $p$.

For the calculation of $\norm{\sigma}$ we claim the following lemma.
\begin{lemma}\label{lem:chasing}
If a vertex of a proper state is a self-crossing (i.e., it is not a crossing between two distinct link components), then it cannot have vertex weight $1$.

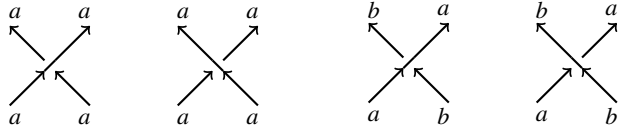
\begin{figure}
\begin{tabular}{cccc}
\begin{tikzpicture}[scale=1,x=1pt,y=1pt]
\draw [->, thick] (13,17) -- (0,30);
\draw [->, thick] (17,17) -- (30,30);
\draw [->, thick] (0,0) -- (13,13);
\draw [->, thick] (30,0) -- (17,13);
\draw [-, thick] (13,13) -- (17,17);

\node at (2,-5) {$\scriptstyle a$};
\node at (28,-5) {$\scriptstyle a$};
\node at (2,35) {$\scriptstyle a$};
\node at (28,35) {$\scriptstyle a$};
\end{tikzpicture}

& \hspace{0.4cm}

\begin{tikzpicture}[scale=1,x=1pt,y=1pt]
\draw [->, thick] (13,17) -- (0,30);
\draw [->, thick] (17,17) -- (30,30);
\draw [->, thick] (0,0) -- (13,13);
\draw [->, thick] (30,0) -- (17,13);
\draw [-, thick] (13,17) -- (17,13);

\node at (2,-5) {$\scriptstyle a$};
\node at (28,-5) {$\scriptstyle a$};
\node at (2,35) {$\scriptstyle a$};
\node at (28,35) {$\scriptstyle a$};
\end{tikzpicture}

&\hspace{0.7cm}

\begin{tikzpicture}[scale=1,x=1pt,y=1pt]
\draw [->, thick] (13,17) -- (0,30);
\draw [->, thick] (17,17) -- (30,30);
\draw [->, thick] (0,0) -- (13,13);
\draw [->, thick] (30,0) -- (17,13);
\draw [-, thick] (13,13) -- (17,17);

\node at (2,-5) {$\scriptstyle a$};
\node at (28,-5) {$\scriptstyle b$};
\node at (2,35) {$\scriptstyle b$};
\node at (28,35) {$\scriptstyle a$};
\end{tikzpicture}

&\hspace{0.4cm}

\begin{tikzpicture}[scale=1,x=1pt,y=1pt]
\draw [->, thick] (13,17) -- (0,30);
\draw [->, thick] (17,17) -- (30,30);
\draw [->, thick] (0,0) -- (13,13);
\draw [->, thick] (30,0) -- (17,13);
\draw [-, thick] (13,17) -- (17,13);

\node at (2,-5) {$\scriptstyle a$};
\node at (28,-5) {$\scriptstyle b$};
\node at (2,35) {$\scriptstyle b$};
\node at (28,35) {$\scriptstyle a$};
\end{tikzpicture}
\end{tabular}
\caption{Typical crossings in a proper state ($a\neq b$)}
\label{fig:tvps}
\end{figure}

\end{lemma}
\begin{proof}[Proof of Lemma~\ref{lem:chasing}]
The proof is done by contradiction:
Assume that a proper state has the crossing of the third (or fourth) diagram in Figure~\ref{fig:tvps}.
If we chase $D$ (not $G(D)$) starting from the crossing along the upper right arc (labeled $a$), then we arrive at the crossing again along the lower right arc (labeled $b$) by the assumption.
However it is impossible because chasing across the crossings of weights $q^{\pm 1}$ or $1$ do \emph{not} alter an arrow-label.
(Consider all crossing types of $D$ depicted in Figure~\ref{fig:tvps}.)
\end{proof}

\begin{corollary}\label{cor:oto}
There is a one-to-one correspondence $\sigma \mapsto \varphi_\sigma$ between the set of all proper states on $D = \bigcup_{j=1}^{m} D_j$ and the set $\mathcal{I}_N^m$.
\end{corollary}
\begin{proof}[Proof of Corollary~\ref{cor:oto}]
By the lemma if $\sigma$ is a proper state, all arcs of $D_j$ have the same arrow-label.
We define this label as $\varphi_\sigma(j)$.
Conversely, from any $\varphi \in \mathcal{I}_N^m$, we get a unique arrow-labeling function by refining $\varphi$.
It is naturally consistent with the conditions (2), (3), (5), (6) in Figure~\ref{fig:splcprojhomfly}.
\end{proof}

If $\sigma$ is a proper symmetric state, then $C_\sigma^\gamma = \bigcup_{I \in \mathcal{I}_N } E_I$ where
\[ E_I := \left\{ \ell \in C_\sigma^\gamma \, : \, \ell \mbox{ has label } I \right\} = \Bigg\{ \ell \in C_\sigma^\gamma \, : \, \ell \mbox{ lies in}\!\! \bigcup_{k \in \varphi_\sigma^{-1}(I)} \!\!\! D_k \Bigg\}. \]
From the previous result $\norm{\sigma} \equiv \sum_{\ell \in C_\sigma^\gamma} \mathrm{label}(\ell)\cdot\mathrm{rot}(\ell) \mbox{ mod } p$, we get
\[ \norm{\sigma} \equiv \sum_{I\in \mathcal{I}_N} I \cdot \sum_{\ell \in E_I} \mathrm{rot}(\ell) = \sum_{I \in \mathcal{I}_N} I \cdot \sum_{k \in \varphi_\sigma^{-1}(I)} \lambda_k = \sum_{j=1}^{m} \varphi_\sigma(j) \lambda_j \: \mbox{ mod } p. \]
If a proper state $\sigma$ is not symmetric, then the quantities $\sum_{j=1}^{m} \varphi_{\zeta^k(\sigma)}(j)\lambda_j$ are the same for all $k \in \{ 0,\ldots, p-1 \}$ thanks to the symmetry of $D$ under $\zeta$.
Therefore $\sum_{k=0}^{p-1} q^{\sum_{j=1}^{m} \varphi_{\zeta^k(\sigma)}(j)\lambda_j} \equiv 0 \mbox{ mod } p$ and so we obtain from Eq.~\eqref{eq:symprop}
\[ \langle D \rangle \equiv \sum_{\sigma \, : \, \mathrm{proper}} q^{\sum_{j=1}^{m} \varphi_\sigma(j)\lambda_j} = \sum_{\varphi \in \mathcal{I}_N^m} q^{\sum_{j=1}^{m} \varphi(j)\lambda_j} \quad \mbox{ mod } (p, q^p - 1). \]
Since $L$ is $p$-periodic, $p$ divides $w(D)$ whence
\[ \mathcal{P}^{(N)}_L(q) = q^{-w(D)\cdot N} \langle D \rangle \equiv \langle D \rangle \quad \mbox{ mod } (q^p -1). \]
as desired.
The proof of Theorem~\ref{thm:maincongm} is now complete.
\end{proof}

\section{More congruences for periodic links}\label{sec:congqpp1}
\noindent In this section we change the ideal $(p, q^p -1)$ in Theorem~\ref{thm:maincongm}.
We first establish the following two (easy) observations;
we have included proofs for the sake of completeness.

\begin{lemma}\label{lem:qmiqpq2p}
Let $c > 0$ be an odd integer and $A$ a commutative ring.
Put $R := A [q^{\pm 1}]$ and $R = R_0 \oplus R_1$ where $R_0 := \bigoplus A q^{2k}$ and $R_1 := \bigoplus A q^{2k+1}$.
If $f \in R_0$ or $f \in R_1$, then $f \in (q^c +1)$ and $f \in (q^{2c} -1)$ whenever $f \in (q^c -1)$.
\end{lemma}

\begin{proof}
Suppose that $f \in R_0$ and $0 \neq f \in (q^c -1)$ with $\deg f = d$.
Then there is $g = \sum_{j= -n}^{d-c} a_j q^j \in R$ such that $f = g(q^c -1)$.
Since $-a_{d-c}q^{d-c} \in R_1$ appears in the product of $g$ and $q^c -1$ it should be canceled by $a_{d-2c}q^{d-c}$ in $a_{d-2c}q^{d-2c}(q^c -1)$.
Thus $d-2c \geq -n$ and $a_{d-c} = a_{d-2c}$ so that
\[ h: = a_{d-c}q^{d-c}(q^c -1) + a_{d-2c}q^{d-2c}(q^c -1) = a_{d-c} q^{d-2c}(q^c -1)(q^c +1) \]
is in $R_0$.
If $f - h = 0$, then the proof is done.
If $f - h \neq 0$, then the reverse induction on $\deg g$ finishes the proof:
The quotient $g'$ for $f - h = g' (q^c -1)$ is $g' = g - (a_{d-c}q^{d-c} + a_{d-2c}q^{d-2c})$.
The case $f \in R_1$ is similar.
\end{proof}

\begin{proposition}\label{prop:pqmipqppq2p}
Let $p$ be an odd prime.
Suppose $R = \Integer [q^{\pm 1}]$ in Lemma~\ref{lem:qmiqpq2p}.
If $f \in R_0$ or $f \in R_1$, then we have $f \in (p, q^p +1)$ and $f \in (p, q^{2p} -1)$ whenever $f \in (p, q^p -1)$.
\end{proposition}

\begin{proof}
Consider the canonical homomorphism
\[ R \rightarrow \overline{R} := \Integer_p [q^{\pm 1}] \, ; \quad g = \sum a_i q^i \mapsto \bar{g} = \sum \bar{a}_i q^i , \]
where $\Integer_p = \{ \bar{0}, \ldots, \overline{p-1} \}$ is the ring of integers modulo $p$.
Suppose $f \in (p, q^p -1) = (p) + (q^p -1)$.
Then $\bar{f} \in (\bar{1}q^p -\bar{1}) \subseteq \overline{R}$.
Since $\bar{f} \in \overline{R}_0$ or $\bar{f} \in \overline{R}_1$ we have $\bar{f} \in (\bar{1}q^p +\bar{1}) \subseteq \overline{R}$ by Lemma~\ref{lem:qmiqpq2p} for $A = \Integer_p$.
This means (from $\overline{R} = R / (p)$) that there exist $f_1 \in (p) \subseteq R$ and $f_2 \in (q^p +1) \subseteq R$ such that $f = f_1 + f_2 \in (p, q^p +1)$.
The case $f \in (p, q^{2p} -1)$ is similar.
\end{proof}

An algebraic result induced from Proposition~\ref{prop:pqmipqppq2p} and Theorem~\ref{thm:maincongm} is:

\begin{corollary}\label{cor:algconscong}
Let $p$ be an odd prime.
Suppose that both the LHS and the RHS of the congruence in Theorem~\ref{thm:maincongm} lie in either $R_0$ or $R_1$ where $R = \Integer [q^{\pm 1}]$.
Then we can replace the ideal $(p, q^p -1)$ by $(p, q^p +1)$ or $(p, q^{2p} -1)$.
\end{corollary}

\begin{figure}
\begin{center}
\begin{tabular}{ccc|cl|l}
$N$ & $r$ & $s$ & $\mathcal{P}^{(N)}_{L}(q)$ & $\quad$($Q_1, Q_2, Q_3$) & Note\\ \hline
even & even & even & odd & = odd + even + even & knots \\
odd & even & even & even & = even + even + even & knots \\
even & odd & odd & even & = odd + even + odd & \\
odd & odd & odd & even & = even + odd + odd & \\
\end{tabular}
\end{center}
\caption{Table for the parities of exponents in $\mathcal{P}^{(N)}_{L}(q)$}
\label{fig:expofR}
\end{figure}

It is known that $r+s$ is always even for each monomial part $c a^r z^s$ ($c,r,s, \in \Integer$) in a homfly polynomial $P_{L}(a,z)$.
By \eqref{eq:defofqi} the term $c a^r z^s$ changes into $c Q_1 Q_2 Q_3$ in $\mathcal{P}^{(N)}_{L}(q)$, where $Q_1 = (q^{N-1} + q^{N-3} + \cdots + q^{-N+1})$, $Q_2 = q^{-rN}$, and $Q_3 = (q - q^{-1})^{s}$.
Thus the parity of exponents in $\mathcal{P}^{(N)}_{L}(q)$ is described in the table of Figure~\ref{fig:expofR}.
On the other hand, the RHS of the congruence in Theorem~\ref{thm:maincongm} depends on $N$ and $\lambda_j$'s.
Note that, whenever $N$ is odd (i.e., all $I \in\mathcal{I}_N$ are even), the parities for the LHS and the RHS coincide automatically.
In this case we can change the ideal by Corollary~\ref{cor:algconscong}.
The preceding argument gives:

\begin{corollary}\label{cor:maincongNoddp2p}
Let $p$ be an odd prime and we keep the same notation as in Theorem~\ref{thm:maincongm}.
If $N$ is odd, then we have up to either modulo $(p, q^p +1)$ or modulo $(p, q^{2p} -1)$
\[  \mathcal{P}^{(N)}_L(q) \equiv \sum_{\varphi \in \mathcal{I}_N^{m}} \prod_{j=1}^{m} q^{\lambda_j \cdot \varphi(j)} . \]
\end{corollary}

The proof of the next theorem (which is, by Corollary~\ref{cor:maincongNoddp2p}, meaningful only when $N$ is even) is based on a different approach.

\begin{theorem}\label{thm:maincongNevenp}
Let $p$ be an odd prime and we keep the same notation as in Theorem~\ref{thm:maincongm}.
Then we have
\[ \mathcal{P}^{(N)}_L(q) \equiv \pm \sum_{\varphi \in \mathcal{I}_N^{m}} \prod_{j=1}^{m} q^{\lambda_j \cdot \varphi(j)} \quad \mbox{ mod } (p, q^p +1). \]
\end{theorem}

\begin{proof}
The proof is a modified version of that of Theorem~\ref{thm:maincongm}.
As in the previous proof, we use the same notation, terminology, and assumption (that $D$ and $G(D)$ are symmetric under $\zeta$).

Let $\sigma$ be an $N$-state on $D$.
If $\sigma$ is not symmetric under $\zeta$, then the congruent states $\sigma, \zeta(\sigma), \ldots, \zeta^{p-1}(\sigma)$ give zero contribution modulo $p$ toward the state model.
If $\sigma$ is symmetric but not proper, then $\langle D | \sigma \rangle$ has a factor $\pm (q - q^{-1})^{p}$ whence zero modulo $(p, q^p +1)$.

We now assume that $\sigma$ is proper.
We see from Lemma~\ref{lem:chasing} that each ``self-crossing'' of a link component is not a vertex of weight $1$ but a vertex of weight $q^{\pm 1}$ depending only on its crossing sign.
Since the number of ``link-crossings'' between two distinct components are even, the parity (modulo $2$) of an integer $k$ for which $\langle D | \sigma \rangle = q^{pk}$ is determined completely by the diagram $D$ (i.e., the crossing signs of the self-crossings in $D$), and does not depend on a specific choice of a proper state $\sigma$ on $D$.
We have now
\[ \langle D \rangle = \sum_{\sigma \, : \, \mathrm{all}} \langle D | \sigma \rangle q^{\norm{\sigma}} \equiv \pm \!\!\! \sum_{\sigma \, : \, \mathrm{proper} \:\: \& \atop \quad\mathrm{symmetric}} q^{\norm{\sigma}} \quad\mbox{ mod } (p, q^p +1). \]

In order to compute $q^{\norm{\sigma}}$, we first decompose $C_{\sigma} = C_{\sigma}^{\gamma} \cup B_{\sigma}^{+} \cup B_{\sigma}^{-}$ for any (not necessarily, proper symmetric) state $\sigma$, where
\[ B_{\sigma}^{\pm} := \{ \ell \in C_{\sigma}\setminus C_{\sigma}^{\gamma} \, : \, \mathrm{rot}(\ell) = \pm 1 \}. \]
Since $\sum_{\ell \in C_{\sigma}} \mathrm{rot}(\ell) = \sum_{\ell \in C_{\sigma}^{\gamma}} \mathrm{rot}(\ell) + \abs{B_{\sigma}^{+}} - \abs{B_{\sigma}^{-}}$ is the rotation number (Whitney degree) of the diagram $D$, it is independent on $\sigma$.
Similarly, since $\sum_{\ell \in C_{\sigma}^{\gamma}} \mathrm{rot}(\ell)$ is the sum $\sum_{j=1}^{m} \lambda_j$ of linking numbers, it is also independent on $\sigma$.
Thus, $\abs{B_{\sigma}^{+}} - \abs{B_{\sigma}^{-}}$ is independent on $\sigma$.
If $\sigma$ is symmetric, then both $\abs{B_{\sigma}^{+}}$ and $\abs{B_{\sigma}^{-}}$ become multiple of $p$ and $t_{\sigma} := (\abs{B_{\sigma}^{+}} + \abs{B_{\sigma}^{-}})/p$ has the same parity with $(\abs{B_{\sigma}^{+}} - \abs{B_{\sigma}^{-}})/p$.
Thus, whenever $\sigma$ is symmetric, we have
\[ \sum_{\ell \in B_{\sigma}^{+} \cup B_{\sigma}^{-}} \mathrm{label}(\ell)\cdot\mathrm{rot}(\ell) = p(\pm I_1 \pm I_2 \pm \cdots \pm I_{t_{\sigma}}) \quad\mbox{ for some } I_j \in \mathcal{I}_N . \]
Note that $(\pm I_1 \pm I_2 \pm \cdots \pm I_{t_{\sigma}})$ has the constant parity for all symmetric $\sigma$.
If $N$ is odd, i.e., $I_j$'s are even, then this fact is trivial.
However, if $N$ is even, i.e., $I_j$'s are odd, then we may use the fact that the parity of $t_{\sigma}$ is independent on a choice of $\sigma$.
Since
\[ \norm{\sigma} = \sum_{\ell \in C_{\sigma}^{\gamma}} \mathrm{label}(\ell)\cdot\mathrm{rot}(\ell) + \sum_{\ell \in B_{\sigma}^{+} \cup B_{\sigma}^{-}} \mathrm{label}(\ell)\cdot\mathrm{rot}(\ell) , \]
we obtain
\[ \sum_{\sigma \, : \, \mathrm{proper} \:\: \& \atop \quad\mathrm{symmetric}} q^{\norm{\sigma}} \equiv \pm \sum_{\sigma \, : \, \mathrm{proper} \:\: \& \atop \quad\mathrm{symmetric}} q^{\sum_{\ell \in C_{\sigma}^{\gamma}} \mathrm{label}(\ell)\cdot\mathrm{rot}(\ell)} \quad\mbox{ mod } (p, q^p +1). \]
As in the proof of Theorem~\ref{thm:maincongm}, we have
\[ \sum_{\ell \in C_{\sigma}^{\gamma}} \mathrm{label}(\ell)\cdot\mathrm{rot}(\ell) = \sum_{I\in \mathcal{I}_N} I \cdot \sum_{\ell \in E_I} \mathrm{rot}(\ell) = \sum_{I \in \mathcal{I}_N} I \cdot \sum_{k \in \varphi_\sigma^{-1}(I)} \lambda_k = \sum_{j=1}^{m} \varphi_\sigma(j) \lambda_j . \]
Since $\sum_{k=0}^{p-1} q^{\sum_{j=1}^{m} \varphi_{\zeta^k(\sigma)}(j)\lambda_j} \equiv 0 \mbox{ mod } p$ for a proper and asymmetric $\sigma$,
\[ \sum_{\sigma \, : \, \mathrm{proper} \:\: \& \atop \quad\mathrm{symmetric}} q^{\sum_{\ell \in C_{\sigma}^{\gamma}} \mathrm{label}(\ell)\cdot\mathrm{rot}(\ell)} \equiv \pm \sum_{\sigma \, : \, \mathrm{proper}} q^{\sum_{j=1}^{m} \varphi_\sigma(j)\lambda_j} = \pm \sum_{\varphi \in \mathcal{I}_N^m} q^{\sum_{j=1}^{m} \varphi(j)\lambda_j} \]
$\mbox{ mod } (p, q^p +1)$.
Finally, since $w(D) \equiv 0 \mbox{ mod } p$, we have
\[ \mathcal{P}^{(N)}_L(q) = q^{-w(D)\cdot N} \langle D \rangle \equiv \pm \langle D \rangle \quad \mbox{ mod } (q^p +1). \]
Gathering the pieces of the above congruences, the proof is done.
\end{proof}

\begin{proposition}\label{prop:Niseven}
Suppose $N$ is even.
If all $\lambda_j$'s are simultaneously odd or even, then the $\pm$-sign of the congruence in Theorem~\ref{thm:maincongNevenp} is determined by the parity of exponents in $\mathcal{P}^{(N)}_L(q)$.
If $L$ is a knot, then the $\pm$-sign is determined by the parity of $\lambda = \lambda_1$.
Moreover, if the sign is ``$+$'', then the ideal $(p, q^p +1)$ in Theorem~\ref{thm:maincongNevenp} can be replaced by $(p, q^{2p} -1)$.
\end{proposition}

\begin{proof}
Let $f,g \in \Integer[q^{\pm 1}]$ be $f := \mathcal{P}^{(N)}_L(q)$ and $g := \sum_{\varphi \in \mathcal{I}_N^{m}} \prod_{j=1}^{m} q^{\lambda_j \cdot \varphi(j)}$.
If all $\lambda_j$'s are odd (resp. even), then the parity of exponents in $g$ is odd (resp. even) because all $I \in \mathcal{I}_N$ are odd.
Consider $\bar{f},\bar{g} \in \Integer_p[q^{\pm 1}]$ as in the proof of Proposition~\ref{prop:pqmipqppq2p}.
Then $\bar{f} \equiv \pm \bar{g} \mbox{ mod } (\bar{1}q^p +\bar{1})$ by Theorem~\ref{thm:maincongNevenp}.
Since $p$ is odd, the $\pm$-sign is therefore determined by the parity of the exponent in $\bar{f}$.
(See Figure~\ref{fig:expofR}.)
Note that if $L$ is a knot, then the exponents in $\bar{f}$ are always odd so that odd $\lambda$ (resp. even $\lambda$) yields the plus sign (resp. minus sign).
Finally, the next proposition completes the proof.
\end{proof}

\begin{proposition}\label{prop:resultant}
Suppose $f \in \Integer[q^{\pm 1}]$.
If $f \in (p, q^p -1) \cap (p, q^p +1)$, then $f \in (p, q^{2p} -1)$.
\end{proposition}

\begin{proof}
Again consider $\bar{f} \in \Integer_p[q^{\pm 1}]$.
Then $\bar{f} \in (\bar{1}q^p -\bar{1}) \cap (\bar{1}q^p +\bar{1})$.
Since the resultant $\mathrm{Res}(\bar{1}q^p -\bar{1}, \bar{1}q^p +\bar{1}) = \bar{2}$ is not zero, there is no common root of $\bar{1}q^p -\bar{1}$ and $\bar{1}q^p +\bar{1}$ in an algebraic closure of $\Integer_p$.
Thus $\bar{f} \in (\bar{1}q^{2p} -\bar{1})$ which implies that $f \in (p, q^{2p} -1)$, as desired.
\end{proof}



\begin{remark}\label{rem:difference}
We do not know whether Theorem~\ref{thm:maincongNevenp} is better than Theorem~\ref{thm:maincongm}.
So far, we could not find an example which turns out to be a non-$p$-periodic link using Theorem~\ref{thm:maincongNevenp} but not possibly by Theorem~\ref{thm:maincongm}.
\end{remark}

\section{Criteria for non-periodicity of knots}\label{sec:criforknot}
\noindent Throughout this section $K$ denotes a knot.
The following result is a case $m=1$ for knots in Theorem~\ref{thm:maincongm} and Theorem~\ref{thm:maincongNevenp}.

\begin{theorem}\label{thm:congknot}
Suppose that $K$ is $p$-periodic with diagram $D$ and the axis $\gamma$, and that $\lambda$ is the linking number between $K$ and $\gamma$.
Then we have
\[ \mathcal{P}^{(N)}_K(q) \equiv \sum_{I \in \mathcal{I}_N} q^{\lambda \cdot I} \quad \mbox{ mod } (p, q^p -1) \]
and
\[ \mathcal{P}^{(N)}_K(q) \equiv \pm \sum_{I \in \mathcal{I}_N} q^{\lambda \cdot I} \quad \mbox{ mod } (p, q^p +1) . \]
\end{theorem}

\begin{corollary}\label{cor:criknot}
If, for any integer $k \in \{ 0, \ldots, p-1 \}$,
\[ \mathcal{P}^{(N)}_K(q) \not\equiv \sum_{I \in \mathcal{I}_N} q^{k\cdot I} \quad \mbox{ mod } (p,q^p -1), \]
then $K$ is not $p$-periodic. 
If, for any integer $k \in \{ 0, \ldots, 2p-1 \}$,
\[ \mathcal{P}^{(N)}_K(q) \not\equiv \pm \sum_{I \in \mathcal{I}_N} q^{k\cdot I} \quad \mbox{ mod } (p,q^p +1), \]
then $K$ is not $p$-periodic. 
\end{corollary}

It is sometimes useful to combine our criteria with existing results like the next corollary suggests.

\begin{corollary}\label{cor:criknot2}
Suppose that we already have a set $\mathcal{C}$ of possible candidate for linking number.
If
\[ \mathcal{P}^{(N)}_K(q) \not\equiv \sum_{I \in \mathcal{I}_N} q^{k\cdot I} \:\: \mbox{ mod } (p,q^p -1) \quad \mbox{ for all } k \in \mathcal{C} \]
or
\[ \mathcal{P}^{(N)}_K(q) \not\equiv \pm \sum_{I \in \mathcal{I}_N} q^{k\cdot I} \:\: \mbox{ mod } (p,q^p +1) \quad \mbox{ for all } k \in \mathcal{C},\]
then $K$ is not $p$-periodic.
\end{corollary}

Fix an odd prime $p$.
Then, by reducing modulo $(p, q^p -1)$, any Laurent polynomial $f = \sum_{i} a_i q^i  \in \Integer [q^{\pm 1}]$ gives rise to a unique representative which we refer to as the ``\textit{normal form}'' of $f$ with respect to $(p, q^p -1)$:
\[ f \equiv \sum_{-p/2 < j < p/2 } b_j q^j \quad \mbox{ mod } (p, q^p -1),\]
where $b_j \in \{ 0,1, \ldots, p-1 \}$ is congruent to $\sum_{i \in j + p\Integer} a_i$ modulo $p$ for each $j$ with $-p/2 < j < p/2$.
When we say two Laurent polynomials in $\Integer [q^{\pm 1}]$ are congruent modulo $(p,q^p-1)$, we mean both the normal forms coincide.

We have the following observation on lower bounds of $p$ for $p$-periodicity.

\begin{corollary}\label{cor:lowerbound}
Suppose that $\mathcal{P}^{(N)}_K(q)$ is not equal to $\sum_{I \in \mathcal{I}_N} q^{k\cdot I}$ for all $k \in \Natural$.
Then there is $n \in \Natural$ such that $K$ is not $p$-periodic for all odd prime $p \geq n$.
\end{corollary}

\begin{proof}
Let $\mathcal{P}$ be the set of all prime divisors of coefficients of $\mathcal{P}^{(N)}_K(q)$.
Set $n := 1+ \max \big( \mathcal{P} \cup \{ 2\deg \mathcal{P}^{(N)}_K(q) \} \big)$.
If $p \geq n$, then $\mathcal{P}^{(N)}_K(q)$ itself is its normal form.
Thus, $\mathcal{P}^{(N)}_K(q)$ cannot be congruent to $\sum_{I \in \mathcal{I}_N} q^{k\cdot I}$ for some $k$ after modulo $(p, q^p - 1)$ by the assumption.
\end{proof}

We next talk about how Theorem~\ref{thm:congknot} itself provides the possible values of linking number.
Suppose $k$ is an integer such that $\abs{\{ k\cdot I \mbox{ mod } p \, : \, I \in \mathcal{I}_N \}} = \abs{\mathcal{I}_N} = N$.
For such $k$, the number of terms in the normal form of $\sum_{I \in \mathcal{I}_N} q^{k \cdot I}$ is $N$, and we address the question which integer $k'$ with $1\leq k' \leq p-1$ satisfies the following identity of sets:
\begin{equation}\label{eq:expch}
\{ k' \cdot I \mbox{ mod } p \, : \, I \in \mathcal{I}_N \} = \{ k \cdot I \mbox{ mod } p \, : \, I \in \mathcal{I}_N \} .
\end{equation}
Important cases where we can easily determine such $k'$ are the following.

\begin{proposition}\label{prop:poss23}
Suppose $N \in \{ 2,3 \}$ and $k \in \Integer$ such that the number of terms in the normal form of $\sum_{I \in \mathcal{I}_N} q^{k \cdot I}$ is $N$.
Then, there are exactly two choices for $k' \in \Integer$ with $1 \leq k' \leq p-1$ satisfying Eq.~\eqref{eq:expch}.
In particular, if $k'$ and $k''$ are such two choices, then $\{ \pm k' \mbox{ mod } p \} = \{ \pm k'' \mbox{ mod } p \}$.
\end{proposition}

\begin{proof}
If $N = 2$ (resp. $N = 3$), then the set $\{ k \cdot I \, : \,  I \in \mathcal{I}_N\}$ is equal to $\{ \pm k \}$ (resp. $\{ 0, \pm 2k \}$).
Since $\mathrm{gcd}(k,p) = 1$ by the assumption, we can uniquely determine $k' \in \Integer$ with $1 \leq k' \leq p-1$ such that $k' \equiv k \mbox{ mod } p$.
Clearly, another choice must be $p - k'$.
\end{proof}

Let $K$ be a knot which we suspect being $p$-periodic.
Suppose we already know $\mathcal{P}^{(N)}_K(q)$ for every $N \in \Natural_{\geq 2}$.
Theoretically, for each $N \in \Natural_{\geq 2}$, we can judge whether or not the congruence $\mathcal{P}^{(N)}_K(q) \equiv \sum_{i\in \mathcal{I}_N} q^{k_N \cdot i} \mbox{ mod } (p,q^p-1)$ holds for some integer $k_N$ depending on $N$ with $0 \leq k_N \leq p-1$.
(Here, the choice for $k_N$ may not be unique.)
Whenever this congruence holds for all elements of a fixed subset $S \subseteq \Natural_{\geq 2}$, we can check further if there is an integer $k$ such that $\{ \pm k \mbox{ mod } p \} = \{\pm k_N \mbox{ mod } p\}$ for all $N \in S$ with some choice of $k_N$, in which case we call $\pm k \mbox{ mod } p$ the ``\textit{possible linking number of $K$ with respect to $S\,$}''.
(Note that the plus-minus sign for $\pm k \mbox{ mod } p$ corresponds to an orientation of $K$.)
On the other hand, if there exists no such $k_N$ for some $N \in \Natural_{\geq 2}$, or there is a subset $S \subseteq \Natural_{\geq 2}$ such that the sets $\{ \pm k_N \mbox{ mod } p\}$ are not the same for all $N \in S$ with some $k_N$, then we can conclude that $K$ is \emph{not} periodic.
By Proposition~\ref{prop:poss23} we often do this procedure with $S = \{ 2,3 \}$.
Of course, the larger the size of $S$ becomes, the more reliable possible linking number we get.

\begin{remark}\label{rem:obtainlk}
There may be many ways to obtain a set $\mathcal{C}$ of possible linking numbers (in order to apply Corollary~\ref{cor:criknot2}).
We can find $\mathcal{C}$ from the previous argument, or from other theorems such as Theorem~\ref{thm:Traczyk}.

Suppose that one method gives $\mathcal{C}$ and the other method gives $\mathcal{C}'$ with $\mathcal{C} \cap \mathcal{C}' = \emptyset$.
In this case we can say that $K$ is not $p$-periodic.
\end{remark}

\end{document}